\documentclass{article}

\usepackage{arxiv}

\usepackage[T1]{fontenc}    
\usepackage{hyperref}       
\usepackage{url}            
\usepackage{CJK}
\usepackage{booktabs}       
\usepackage{amsfonts}       
\usepackage{nicefrac}       
\usepackage{microtype}      
\usepackage{lipsum}
\usepackage{graphicx}
\usepackage{amsmath}
\usepackage[all]{xy}
\usepackage{amsthm}
\usepackage{eucal}
\usepackage{framed}

\newtheorem{theorem}{Theorem}
\theoremstyle{definition}

\newtheorem{example}[theorem]{Example}

\newtheorem{lemma}{Lemma}

\newtheorem{corollary}{Corollary}

\newtheorem{conjecture}{Conjecture}

\title{The summation of infinite partial fraction decomposition I: some formulae related to the Hurwitz zeta function}

\author{
  Xiaowei Wang\thanks{This paper is written in Dec 2020}
   \\
}

\begin{document}
\maketitle

\begin{abstract}
In this paper we establish a new summation method by expanding  $\prod_{k}(1-\frac{z}{a_{k}})^{-1}$ with two approaches: the Taylor expansion and the infinite partial fraction decomposition. Here we focus on the case when $a_{k}$ is arithmetic sequence. By this summation we obtain many equalities involve Hurwitz zeta function and Gammma function. 
\end{abstract}


\section{Introduction}
Let $S$ be the series of rational functions $S=\sum\frac{P(n)}{Q(n)}$ with $\deg(Q)\geq \deg(P)+2$. By decomposing $\frac{P}{Q}$ into partial fractions, $S$ can be rewritten as finite linear combination of Hurwitz zeta function $\zeta(m,a)=\sum_{k=0}^{\infty}\frac{1}{(a+k)^m}$ with positive interger $m$. In this paper, we construct another series representation of $\zeta(m,a)$ by discuss two kinds of expansions of the infinite product $\prod_{k=0}^{\infty}(1-(\frac{z}{a+k})^m)$.\\
In fact, the correspondence is natural. In general, there is some relation between $\sum_{k=0}^{\infty}\frac{1}{a_{k}^m}$ and $\sum_{k=0}^{\infty}\frac{1}{F'(a_{k})}\frac{1}{a_{k}^{m+1}}$, where $F(z)$ is the product $\prod_{k=0}^{\infty}(1-(\frac{z}{a_{k}})^m)$. Basically, the relation is derived by the infinite partial fraction decomposition of $\frac{1}{F(z)}$. And the correspondence is not only valid for infinite series, but also valid for finite sum. Therefore the convergence can be discussed easily. In this we paper we study this relation and focus on the case when $a_{k}=a+k$, namely, the problem involves the Hurwitz zeta function.\\
Throughout the paper we denote the coefficient of $z^n$ in the Laurent expansion of $f(z)$ around $0$ by $[f(z)]_{n}$. The notation $F'(a_{n})$ refers to $\lim_{z\rightarrow a_{n}}F'(z)$.

\section{Prelimilary}

\begin{lemma}\label{lemma1}(Homogeneous partial fraction decomposition)\\
Let $a_{1},...,a_{n}$ be distinct complex numbers, $x\in \mathbb{C}\backslash \{a_{1},...,a_{n}\} $, then there exist $\mu_{1},...,\mu_{n}\in \mathbb{C}$ such that following identity is true,\\
\begin{equation}\label{pfd1}
  \prod_{i=1}^{n}\frac{1}{x-a_{i}}=\sum_{i=1}^{n}\frac{\mu_{i}}{x-a_{i}}
\end{equation}
where $\mu_{i}=\prod_{j=1,j\neq i}^{n}\frac{1}{a_{i}-a_{j}}$. Further, The relation $\sum_{i=1}^{n}\mu_{i}=0$ holds.\\
\end{lemma}
\begin{proof}
The proof is easy, one can consider the Lagrange's interpolation formula. For more details, See author's paper \cite{xiaowei}.
\end{proof}

\begin{lemma}\label{lemma2}(One point lifting)\\
Suppose that $N, L\in \mathbb{Z}_{\geq 0}$ and $a_{n}$ are complex numbers that differ from each other, then the following identity holds
\begin{equation*}
  \frac{1}{x^L}\prod_{n=1}^{N}\frac{1}{x-a_{n}}=\sum_{j=1}^{L}\sum_{n=1}^{N}\frac{-\mu_{n}}{a_{n}^{L-j+1}}\frac{1}{x^{j}}+\sum_{n=1}^{N}\frac{\mu_{n}}{a_{n}^{L}}\frac{1}{x-a_{n}}
\end{equation*}
where
\begin{equation*}
\mu_{n}=\prod_{s=1,s\neq n}^{N}\frac{1}{a_{n}-a_{s}}
\end{equation*}
\end{lemma}
\begin{proof}
Firstly one can see the case $L=0$ is exactly Lemma \ref{pfd1}, hence the the identity is true for $\ell=0$. Following we assume that $L\in \mathbb{Z}_{\geq 1}$. Suppose that $z_{0},z_{1},...,z_{N}$ be complex variables that take distinct values. Let $H(z_{0},z_{1},...,z_{n})=\prod_{n=0}^{N}\frac{1}{x-z_{n}}$. Then by Lemma 1 we have the identity\\
\begin{equation*}
H(z_{0},z_{1},...,z_{n})=\sum_{n=0}^{N}\frac{\lambda_{n}}{x-z_{n}}
\end{equation*}
Where $\lambda_{n}=\prod_{s=0,s\neq n}^{N}\frac{1}{z_{n}-z_{s}}$. Now we taking partial derivatives of $H$ with respect to $z_{0}$. There are two ways to carry out. On the one hand,\\
\begin{equation}\label{lemma2.1}
\frac{\partial^{l}H}{\partial z_{0}^{l}}=\frac{l!}{(x-z_{0})^{l+1}}\prod_{n=1}^{\infty}\frac{1}{x-z_{n}}
\end{equation}
On the other hand,\\
\begin{align*}
  \frac{\partial^{l}H}{\partial z_{0}^{l}}&=\frac{\partial^{l}}{\partial z_{0}^{l}}(\lambda_{0}\frac{1}{x-z_{0}})+\sum_{n=1}^{N}\frac{1}{x-z_{n}}\frac{\partial^{l}\lambda_{n}}{\partial z_{0}^{l}}\\
  &=\sum_{j=0}^{l}\binom{l}{j}\frac{\partial^{l-j}\lambda_{0}}{\partial z_{0}^{l-j}}\frac{ j!}{(x-z_{0})^{j+1}}+\sum_{n=1}^{N}\frac{1}{x-z_{n}}\frac{\partial^{l}\lambda_{n}}{\partial z_{0}^{l}}\\
\end{align*}
Recall that $\lambda_{0}=\prod_{s=1}^{N}\frac{1}{z_{0}-z_{s}}$. reapplying Lemma \ref{pfd1}\\
\begin{equation*}
  \lambda_{0}=\sum_{n=1}^{N}\frac{M_{n}}{z_{0}-z_{n}}
\end{equation*}
where $M_{n}=\prod_{s=1,s\neq n}^{N}\frac{1}{z_{n}-z_{s}}$. Therefore\\
\begin{equation*}
  \frac{\partial^{l-j}\lambda_{0}}{\partial z_{0}^{l-j}}=\sum_{n=1}^{N}M_{n}\frac{(-1)^{l-j} (l-j)!}{(z_{0}-z_{n})^{l-j+1}}
\end{equation*}
Further, for $n=1,...,N$\\
\begin{equation*}
  \frac{\partial^{l}\lambda_{n}}{\partial z_{0}^{l}}=\frac{l!M_{n}}{(z_n-z_{0})^{l+1}}
\end{equation*}
Now compare to (\ref{lemma2.1}) we obtain\\
\begin{equation*}
  \frac{1}{(x-z_{0})^{l+1}}\prod_{n=1}^{\infty}\frac{1}{x-z_{n}}=\sum_{j=0}^{l}\sum_{n=1}^{N}\frac{M_{n}(-1)^{l-j} (l-j)!}{(z_{0}-z_{n})^{l-j+1}}\frac{ 1}{(x-z_{0})^{j+1}}+\sum_{n=1}^{N}\frac{1}{x-z_{n}}\frac{M_{n}}{(z_n-z_{0})^{l+1}}
\end{equation*}
Lastly let $L=l+1$, $z_{0}=0$, and $z_{i}=a_{i}$, $\mu_{n}=M_{n}|_{z_{i}=a_{i}}$ for $i=1,...,N$ respectively, we obtain what required.
\end{proof}
\newpage

\section{Main Results}
In the last section we introduce some lemmas about partial fraction decomposition. In follow we discuss the partial fraction expansion of the infinite product $(z^{L}F(z))^{-1}=z^{-L}\prod_{n=1}^{\infty}(1-\frac{z}{a_{n}})^{-1}$. We then show that the rearrangement to the power series of $z$ remains invariant for different $L$.\\

\begin{theorem}\label{thm1}(\textbf{Partial Fraction Summation})\\
Let $(a_{n})$ be a sequence in $\mathbb{C}\backslash\{0\}$ and satisfy that: 1, $a_{n}$ differ from each other for all $n$; 2, $\sum_{n=1}^{\infty}\frac{1}{a_{n}}$ absolutely converges. Define\\
\begin{equation*}
    F(z)=\prod_{n=1}^{\infty}(1-\frac{z}{a_{n}})
\end{equation*}
then following identity holds for $k=0,1,2,...$.\\
\begin{equation*}
[\frac{1}{F(z)}]_{k}=\sum_{n=1}^{\infty}\frac{-1}{F'(a_{n})}\frac{1}{a_{n}^{k+1}}
\end{equation*}
On the other hand, for $L=0,1,2,...$, $\frac{1}{F(z)}$ has the following partial fraction decompositions\\
\begin{equation*}
    \frac{1}{F(z)}=[\frac{1}{F(z)}]_{0}+[\frac{1}{F(z)}]_{1}z+...+[\frac{1}{F(z)}]_{L-1} z^{L-1}+\sum_{n=1}^{\infty}\frac{z^L}{F'(a_{n})a_{n}^{L}(z-a_{n})}
\end{equation*}
\textbf{For given $L$, we call $L$ the order of this expansion}.\\
\end{theorem}
\begin{proof}
Let $M=\inf(|a_{n}|)$, in the proof we always assume that $|z|<M$. Consider the expansion of $\frac{1}{F(z)}$, namely,\\
\begin{equation*}
    \frac{1}{F(z)}=\exp(-\sum_{n=1}^{\infty}\log(1-\frac{z}{a_{n}}))=\exp(\sum_{k=1}^{\infty} \sum_{n=1}^{\infty} \frac{1}{a_{n}^k}\frac{z^k}{k})
\end{equation*}
Denote $[\frac{1}{F(z)}]_{k}$ by $c_{k}$, it can be computed by expanding the right hand side in the above formula. Then\\
\begin{equation}\label{thm1.1}
    \frac{1}{z^{L}F(z)}=\frac{1}{z^{L}}(c_{0}+c_{1}z+c_{2}z^2+...)
\end{equation}
On the other hand, consider the another expansion via Lemma \ref{lemma2} and let $N\rightarrow \infty$, that is\\
\begin{equation}\label{thm1.2}
     \frac{1}{z^{L}F(z)}=\frac{1}{z^L}(\sum_{n=1}^{\infty}\frac{-\lambda_{n}}{a_{n}})+...+\frac{1}{z^2}(\sum_{n=1}^{\infty}\frac{-\lambda_{n}}{a_{n}^{L-1}})+\frac{1}{z}(\sum_{n=1}^{\infty}\frac{-\lambda_{n}}{a_{n}^L})+\sum_{n=1}^{\infty}\frac{\lambda_{n}}{a_{n}^{L}(z-a_{n})} 
\end{equation}
Where $\lambda_{n}=-a_{n}\prod_{s=1,s\neq n}^{\infty}\frac{a_{s}}{a_{s}-a_{n}}$. Further, we can reformulate $\lambda_{n}$ as following\\
\begin{equation*}
    \lambda_{n}=-a_{n}\lim_{z-a_{n}\rightarrow 0}\frac{a_{n}-z}{a_{n}}\prod_{s=1}^{\infty}\frac{a_{s}}{a_{s}-z}=\frac{1}{F'(a_{n})}
\end{equation*}
Now consider the expansion of $\sum_{n=1}^{\infty}\frac{\lambda_{n}}{a_{n}^{L}(z-a_{n})}$, due to the absolute convergence of $\sum\frac{1}{a_{n}}$, for $|z|<M$ we have \\
\begin{equation*}
    \sum_{n=1}^{\infty}\frac{\lambda_{n}}{a_{n}^{L}(z-a_{n})}=\sum_{n=1}^{\infty}\frac{-\lambda_{n}}{a_{n}^{L+1}}+\sum_{n=1}^{\infty}\frac{-\lambda_{n}}{a_{n}^{L+2}}z+\sum_{n=1}^{\infty}\frac{-\lambda_{n}}{a_{n}^{L+3}}z^2+...
\end{equation*}
Then equation (\ref{thm1.2}) becomes\\
\begin{equation}
    \frac{1}{z^{L}F(z)}=\frac{1}{z^L}(\sum_{n=1}^{\infty}\frac{-\lambda_{n}}{a_{n}}+\sum_{n=1}^{\infty}\frac{-\lambda_{n}}{a_{n}^{2}}z+...+\sum_{n=1}^{\infty}\frac{-\lambda_{n}}{a_{n}^{L+1}}z^{L}+...)
\end{equation}
Compare to (\ref{thm1.1}) we have for $k=0,1,2,3,...$, the equalities hold\\
\begin{equation*}
    c_{k}=\sum_{n=1}^{\infty}\frac{-1}{F'(a_{n})}\frac{1}{a_{n}^{k+1}}
\end{equation*}
\end{proof}

\begin{corollary}
As a direct corollary of theorem \ref{thm1}, \\
\begin{align*}
   & \sum_{n=1}^{\infty}\frac{-1}{F'(a_{n})}\frac{1}{a_{n}}\equiv 1 \\
   & \sum_{n=1}^{\infty}\frac{-1}{F'(a_{n})}\frac{1}{a_{n}^2}=\sum_{n=1}^{\infty}\frac{1}{a_{n}}\\
   &\sum_{n=1}^{\infty}\frac{-1}{F'(a_{n})}\frac{1}{a_{n}^3}=\frac{1}{2}((\sum_{n=1}^{\infty}\frac{1}{a_{n}})^2+\sum_{n=1}^{\infty}\frac{1}{a_{n}^2})\\
   &...
\end{align*}
\end{corollary}

\begin{theorem}\label{thm2}(\textbf{Partial fraction summation for symmetric products})\\
Let $(a_{n})$ be a sequence in $\mathbb{C}\backslash\{0\}$ and satisfy that: 1, $a_{n}$ differ from each other for all $n$; 2, $\sum_{n=1}^{\infty}\frac{1}{a_{n}^2}$ absolutely converges. Define\\
\begin{equation*}
    F(z)=\prod_{n=1}^{\infty}(1-(\frac{z}{a_{n}})^2)
\end{equation*}
then following identity holds for $k=0,1,2,...$.\\
\begin{equation*}
[\frac{1}{F(z)}]_{2k}=\sum_{n=1}^{\infty}\frac{-2}{F'(a_{n})}\frac{1}{a_{n}^{2k+1}}
\end{equation*}
On the other hand, for $L=0,2,4,...$, $\frac{1}{F(z)}$ has the following partial fraction decompositions\\
\begin{equation}\label{thm2.11}
    \frac{1}{F(z)}=[\frac{1}{F(z)}]_{0}+[\frac{1}{F(z)}]_{2}z^2+...+[\frac{1}{F(z)}]_{L-2} z^{L-2}+\sum_{n=1}^{\infty}\frac{2z^L}{F'(a_{n})a_{n}^{L-1}(z^2-a_{n}^2)}
\end{equation}
for $L=1,3,5,...$, $\frac{1}{F(z)}$ has the following partial fraction decompositions\\
\begin{equation}\label{thm2.12}
    \frac{1}{F(z)}=[\frac{1}{F(z)}]_{0}+[\frac{1}{F(z)}]_{2}z^2+...+[\frac{1}{F(z)}]_{L-1} z^{L-1}+\sum_{n=1}^{\infty}\frac{2z^{L+1}}{F'(a_{n})a_{n}^{L}(z^2-a_{n}^2)}
\end{equation}
\end{theorem}
\begin{proof}
By Theorem \ref{thm1}, one can let $F(z)=\prod_{n=1}^{\infty}(1-\frac{z}{b_{n}})$ where $b_{2n-1}=a_{n}, b_{2n}=-a_{n}$. In order the find $F'(b_{n})$, just need to note that $F(z)=f(z)f(-z)$, where $f(z)=\prod_{n=1}^{\infty}(1-\frac{z}{a_{n}})$ ( This infinite product may be divergent, one can use finite product to approximate ), Then\\
\begin{align*}
    &F'(b_{2n-1})=(f(z)f(-z))'|_{z=a_{n}}=f'(a_{n})f(-a_{n})-f(a_{n})f'(-a_{n})=f'(a_{n})f(-a_{n})\\
    &F'(b_{2n})=(f(z)f(-z))'|_{z=-a_{n}}=f'(-a_{n})f(a_{n})-f(-a_{n})f'(a_{n})=-f'(a_{n})f(-a_{n})\\
\end{align*}
That is $F'(b_{2n-1})=-F'(b_{2n})$. By Theorem \ref{thm1} we have\\
\begin{align*}
    [\frac{1}{F(z)}]_{2k-1}&=\sum_{n=1}^{\infty}\frac{-1}{F'(b_{n})}\frac{1}{b_{n}^{2k}}\\
    &=\sum_{n=1}^{\infty}\frac{-1}{F'(b_{2n-1})}\frac{1}{b_{2n-1}^{2k}}+\sum_{n=1}^{\infty}\frac{-1}{F'(b_{2n})}\frac{1}{b_{2n}^{2k}}\\
    &=\sum_{n=1}^{\infty}\frac{-1}{F'(b_{2n-1})}\frac{1}{a_{n}^{2k}}+\sum_{n=1}^{\infty}\frac{-1}{F'(b_{2n})}\frac{1}{a_{n}^{2k}}\\
    &=0
\end{align*}
This is trivial since $F(z)$ is even function.\\ 
\begin{align*}
    [\frac{1}{F(z)}]_{2k}&=\sum_{n=1}^{\infty}\frac{-1}{F'(b_{n})}\frac{1}{b_{n}^{2k+1}}\\
    &=\sum_{n=1}^{\infty}\frac{-1}{F'(b_{2n-1})}\frac{1}{b_{2n-1}^{2k+1}}+\sum_{n=1}^{\infty}\frac{-1}{F'(b_{2n})}\frac{1}{b_{2n}^{2k+1}}\\
    &=\sum_{n=1}^{\infty}\frac{-1}{F'(b_{2n-1})}\frac{1}{a_{n}^{2k+1}}-\sum_{n=1}^{\infty}\frac{-1}{F'(b_{2n})}\frac{1}{a_{n}^{2k+1}}\\
    &=\sum_{n=1}^{\infty}\frac{-2}{F'(a_{n})}\frac{1}{a_{n}^{2k+1}}
\end{align*}
The remaining partial fraction decomposition formulae (\ref{thm2.11}), (\ref{thm2.12}) are straightforward, we omit the computation.
\end{proof}

Let $a\in \mathbb{C}\backslash \{0,-1,-2,...\}$, define\\
\begin{equation*}
    g_{a}(z):=\frac{\Gamma(a+z)\Gamma(a-z)}{\Gamma(a)^2}
\end{equation*}
Then $g_{a}(z)$ has the infinite product representation\\
\begin{equation}\label{ex1.1}
    g_{a}(z)=\prod_{k=0}^{\infty}(1-(\frac{z}{a+k})^2)^{-1}
\end{equation}
It can be proved easily by using the formula\\
\begin{equation}\label{gammafml}
\frac{\Gamma(a+z)}{\Gamma(a)}=\frac{a e^{-\gamma z}}{a+z}\prod_{n=1}^{\infty}(1+\frac{z}{a+n})^{-1}e^{\frac{z}{n}}
\end{equation}
Now by Theorem \ref{thm2} we have the partial fraction decomposition of $g_{a}(z)$ of order $1$\\
\begin{equation}\label{ex1.2}
\Gamma(a+z)\Gamma(a-z)=\Gamma(a)^2+\sum_{k=0}^{\infty}\frac{2(-1)^{k+1}\Gamma(2a+k)}{(a+k)k!}\frac{z^2}{z^2-(a+k)^2}
\end{equation}
In order to keep the convergence, we shall keep $\Re{(a)}<2$ and $|z|<|a|$.
Further, it's easy to obtain $[\Gamma(a+z)\Gamma(a-z)]_{2}$ by (\ref{ex1.1}), namely $[\Gamma(a+z)\Gamma(a-z)]_{2}=\zeta(2,a)$. On the other hand, by finding the Taylor expansion of (\ref{ex1.2}), one has\\
\begin{corollary}(\textbf{The PFS representation of $\zeta(2,a)$})\\
\begin{equation}\label{ex1.3}
   \zeta(2,a)= \sum_{k=0}^{\infty}\frac{1}{(a+k)^2}=\sum_{k=0}^{\infty}\frac{2(-1)^{k}\Gamma(2a+k)}{\Gamma(a)^2 k!(a+k)^{3}}, (\Re(a)<2, a\neq 0,-\frac{1}{2},-1,-\frac{3}{2},...)
\end{equation}
\end{corollary}

\begin{theorem}\label{thm3}(\textbf{Partial fraction summation for cyclotomic product})\\
Let $(a_{n})$ be a sequence in $\mathbb{C}\backslash\{0\}$ and satisfy that:\\
I, $a_{n}$ differ from each other for all $n$;\\
II, $\sum_{n=1}^{\infty}\frac{1}{a_{n}^m}$ absolutely converges.\\
Define\\
\begin{equation*}
    F(z)=\prod_{n=1}^{\infty}(1-(\frac{z}{a_{n}})^m)
\end{equation*}
then following identity holds for $J=0,1,2,...$.\\
\begin{equation}\label{thm3.1}
[\frac{1}{F(z)}]_{J}=\begin{cases}
0, \text{ if }\mod(J,m)\neq 0\\
\sum_{n=1}^{\infty}\frac{-m}{F'(a_{n})}\frac{1}{a_{n}^{J+1}}, \text{ if } \mod(J,m)= 0\\
\end{cases}
\end{equation}
\end{theorem}
\begin{proof}
The proof is quite similar to Theorem \ref{thm2}. For $k=1,2,..$, let $k=m(n-1)+r$ and \\
\begin{equation*}
   b_{k}=b_{m(n-1)+r}=a_{n}\omega_{m}^{r-1}, r=1,2,...,m
\end{equation*}
where $\omega_{m}=e^{\frac{2i\pi}{m}}$ is the $m-$th root of unity. Then\\
\begin{equation*}
    F(z)=\prod_{k=1}^{\infty}(1-\frac{z}{b_{k}})
\end{equation*}
On the othe hand, let\\
\begin{equation*}
    f_{r,N}(z)=\prod_{n=1}^{N}(1-\frac{z}{a_{n}\omega_{m}^{r-1}}), r=1,2,...,m
\end{equation*}
Then $f_{r,N}(z)=f_{1,N}(z/\omega_{m}^{r-1})$, we denote $f_{1,N}(.)$ directly by $f_{N}(.)$. Then $f_{r,N}(a_{n}\omega_{m}^{r-1})=f_{N}(a_{n})=0$ for $n\leq N$.
To compute $F'(b_{k})=F'(b_{m(n-1)+r})$: \\ 
\begin{align*}
    F'(b_{m(n-1)+r})&=\lim_{N\rightarrow \infty} ( f_{1,N}(z) f_{2,N}(z)... f_{m,N}(z))'|_{z=a_{n}\omega_{m}^{r-1}}\\
    &=\lim_{N\rightarrow \infty}f_{N}(z)f_{N}(\frac{z}{\omega_{m}})...f_{N}(\frac{z}{\omega_{m}^{m-1}})(\frac{f_{N}'(z)}{f_{N}(z)}+\frac{1}{\omega_{m}}\frac{f_{N}'(z/\omega_{m})}{f_{N}(z/\omega_{m})}+...+\frac{1}{\omega_{m}^{m-1}}\frac{f_{N}'(z/\omega_{m}^{m-1})}{f_{N}(z/\omega_{m}^{m-1})})|_{z=a_{n}\omega_{m}^{r-1}}\\
    &=\frac{1}{\omega_{m}^{r-1}}f'(a_{n})f(a_{n}\omega_{m})...f(a_{n}\omega_{m}^{m-1})
\end{align*}
One can implies that for all $n=1,2,3,...$ and $r=1,2,...,m$,\\
\begin{equation*}
  F'(b_{m(n-1)+r})=\frac{1}{\omega_{m}^{r-1}}F'(b_{m(n-1)+1})  
\end{equation*}
According to Theorem \ref{thm1}, we have\\
\begin{align*}
    [\frac{1}{F(z)}]_{J}&=\sum_{k=1}^{\infty}\frac{-1}{F'(b_{k})}\frac{1}{b_{k}^{J+1}}\\
    &=\sum_{r=1}^{m}\sum_{n=1}^{\infty}\frac{-1}{F'(b_{m(n-1)+r})}\frac{1}{(a_{n}\omega_{m}^{r-1})^{J+1}}\\
    &=\sum_{r=1}^{m}\frac{1}{\omega_{m}^{J(r-1)}}\sum_{n=1}^{\infty}\frac{-1}{F'(b_{m(n-1)+1})}\frac{1}{a_{n}^{J+1}}
\end{align*}
Therefore\\
\begin{equation*}
[\frac{1}{F(z)}]_{J}=\begin{cases}
0, \text{ if }\mod(J,m)\neq 0\\
\sum_{n=1}^{\infty}\frac{-m}{F'(a_{n})}\frac{1}{a_{n}^{J+1}}, \text{ if } \mod(J,m)= 0\\
\end{cases}
\end{equation*}
\end{proof}
\begin{corollary}(\textbf{The PFS representation of $\zeta(m,a)$})\\
For $m=2,3,4,...$, following identity holds,\\
\begin{equation}\label{co3.1}
    \zeta(m,a)=\sum_{k=0}^{\infty}\frac{m(-1)^{k}\prod_{r=1}^{m-1}\Gamma(a-\omega_{m}^{r}(a+k))}{\Gamma(a)^m k! (a+k)^{m+1}}
\end{equation}
\end{corollary}
\begin{proof}
In the Theorem \ref{thm3} for convenience let $k=n-1$, let $a_{n}=a-1+n=a+k$, then $\frac{1}{F(z)}=\prod_{k=0}^{\infty}(1-(\frac{z}{a+k})^m)^{-1}$, it easy to see that $[\frac{1}{F(z)}]_{m}=\zeta(m,a)$, now according to Theorem \ref{thm3}, it remains to find $F'(a+k)$. Note that\\
\begin{equation*}
    F(z)=\frac{\Gamma(a)^m}{\prod_{r=0}^{m-1}\Gamma(a-\omega_{m}^{r}z)}
\end{equation*}
this can be easily derived by the formula (\ref{gammafml}). After some computations\\
\begin{equation*}
    F'(z)=\frac{\Gamma(a)^m}{\prod_{r=0}^{m-1}\Gamma(a-\omega_{m}^{r}z)}\sum_{r=0}^{m-1}\omega_{m}^{r}\psi(a-\omega_{m}^{r}z)
\end{equation*}
Note that $\lim_{z\rightarrow -k}\frac{\psi(z)}{\Gamma(z)}=(-1)^{k-1}k!$, hence\\
\begin{equation*}
    F'(a+k)=\frac{\Gamma(a)^m k!(-1)^{k-1}} {\prod_{r=1}^{m-1}\Gamma(a-\omega_{m}^{r}(a+k))}
\end{equation*}
Finally by (\ref{thm3.1}) we obtain what required, namely\\
\begin{equation*}
    \zeta(m,a)=\sum_{k=0}^{\infty}\frac{-m}{F'(a+k)}\frac{1}{(a+k)^{m+1}}=\sum_{k=0}^{\infty}\frac{m(-1)^{k}\prod_{r=1}^{m-1}\Gamma(a-\omega_{m}^{r}(a+k))}{\Gamma(a)^m k! (a+k)^{m+1}}
\end{equation*}
\end{proof}

\begin{corollary}(\textbf{The PFS representation of $\zeta(m)$})\\
Let $a=1$, we have then\\
\begin{equation*}
    \zeta(m)=\sum_{n=1}^{\infty}\frac{m(-1)^{n-1}\Gamma(1-\omega_{m}n)...\Gamma(1-\omega_{m}^{m-1}n)}{n!n^m}
\end{equation*}
As a special case, after some simplification, we have a series representation for Apr\`{e}y's constant\\
\begin{equation*}
\zeta(3)=\sum_{n=1}^{\infty}\frac{3(-1)^{n-1}\Gamma(\frac{1+\sqrt{3}}{2}n)\Gamma(\frac{1-\sqrt{3}}{2}n)}{n\cdot n!}
\end{equation*}
\end{corollary}

\begin{theorem}\label{thm4}(\textbf{The partial fraction summation for general product)}\\
Let\\
\begin{equation*}
  F(z)=\prod_{k=0}^{\infty}\prod_{m=1}^{M}(1-\frac{z}{a_{m,k}})
\end{equation*}
then formally we have
\begin{equation*}
  [\frac{1}{F(z)}]_{\ell}=\sum_{k=0}^{\infty}\sum_{m=1}^{M}\frac{-1}{a_{m,k}^{\ell+1}F'(a_{m,k})}
\end{equation*}
especially, if $F(z)$ normally converges in $|z|<M$ for some $M>0$, then
\begin{align*}
   &[\frac{1}{F(z)}]_{0}= \sum_{k=0}^{\infty}\sum_{m=1}^{M}\frac{-1}{a_{m,k}F'(a_{m,k})}=1 \\
&[\frac{1}{F(z)}]_{1}= \sum_{k=0}^{\infty}\sum_{m=1}^{M}\frac{-1}{a_{m,k}^2 F'(a_{m,k})}=\sum_{k=0}^{\infty}\sum_{m=1}^{M}\frac{1}{a_{m,k}}
\end{align*}
\end{theorem}
\begin{proof}
The proof is similar to the proof of Theorem \ref{thm3}. We left it to the reader.\\
\end{proof}

\begin{example}
Let $F(z)=\frac{\Gamma(\frac{1}{2})}{\Gamma(1+\frac{z}{2})\Gamma(\frac{1}{2}-\frac{z}{2})}$, then we can reformulate $F(z)$ as\\
\begin{equation*}
  F(z)=\prod_{k=1}^{\infty}(1-\frac{z}{a_{k}})(1-\frac{z}{b_{k}})
\end{equation*}
where $a_{k}=2k-1$, $b_{k}=-2k$. It's not hard to obtain\\
\begin{equation*}
  F'(z)=\frac{\sqrt{\pi}}{2\Gamma(1+\frac{z}{2})\Gamma(\frac{1}{2}-\frac{z}{2})}(\psi(\frac{1}{2}-\frac{z}{2})-\psi(1+\frac{z}{2}))
\end{equation*}
It turns out that\\
\begin{equation*}
F'(a_{k})=-F'(b_{k}) =\frac{\sqrt{\pi}}{2}\frac{(-1)^k (k-1)!}{\Gamma(k+\frac{1}{2})}
\end{equation*}
By the Theorem \ref{thm4}\\
\begin{align*}
   & \sum_{k=1}^{\infty}\frac{(-1)^{k-1} \Gamma(k+\frac{1}{2})}{(k-1)!}(\frac{1}{2k-1}+\frac{1}{2k})= \frac{\sqrt{\pi}}{2}\\
   & \sum_{k=1}^{\infty}\frac{(-1)^{k-1} \Gamma(k+\frac{1}{2})}{(k-1)!}(\frac{1}{(2k-1)^2}-\frac{1}{(2k)^2})= \frac{\sqrt{\pi}\log 2}{2}\\
   & \sum_{k=1}^{\infty}\frac{(-1)^{k-1} \Gamma(k+\frac{1}{2})}{(k-1)!}(\frac{1}{(2k-1)^3}+\frac{1}{(2k)^3})= \frac{\sqrt{\pi}(\log^2 2+\zeta(2))}{4}
\end{align*}
Note that $\frac{\Gamma(k+\frac{1}{2})}{(k-1)!}=\sqrt{\pi}\frac{(2k-1)!!}{(2k-2)!!}$, in fact\\
\begin{align*}
   & \sum_{k=1}^{\infty}(-1)^{k-1}\frac{ (2k-1)!!}{(2k-2)!!}(\frac{1}{2k-1}+\frac{1}{2k})= \frac{1}{2}\\
   & \sum_{k=1}^{\infty}(-1)^{k-1}\frac{ (2k-1)!!}{(2k-2)!!}(\frac{1}{(2k-1)^2}-\frac{1}{(2k)^2})= \frac{\log 2}{2}\\
   & \sum_{k=1}^{\infty}(-1)^{k-1}\frac{ (2k-1)!!}{(2k-2)!!}(\frac{1}{(2k-1)^3}+\frac{1}{(2k)^3})= \frac{\log^2 2+\zeta(2)}{2}
\end{align*}
\end{example}
More general, we have
\begin{example}
Assume that $0\leq \Re(a),\Re(b)<1$ and $a,b\neq 0$. Let $a_{k}=a+k$, $b_{k}=-b-k$, $k=0,1,2,...$. Consider\\
\begin{equation*}
  F(z)=\prod_{k=0}^{\infty}(1-\frac{z}{a_{k}})(1-\frac{z}{b_{k}})=\frac{\Gamma(a)\Gamma(b)}{\Gamma(a-z)\Gamma(b+z)}
\end{equation*}
One can show that\\
\begin{equation*}
F'(z)=\frac{\Gamma(a)\Gamma(b)}{\Gamma(a-z)\Gamma(b+z)}(\psi(a-z)-\psi(b-z))
\end{equation*}
That is\\
\begin{equation*}
  \frac{1}{F'(a_{k})}=-\frac{1}{F'(b_{k})}=\frac{(-1)^{k-1}\Gamma(a+b+k)}{k!\Gamma(a)\Gamma(b)}
\end{equation*}
Therefore\\
\begin{equation*}
[\frac{1}{F(z)}]_{J}=\sum_{k=0}^{\infty}\frac{(-1)^{k}\Gamma(a+b+k)}{k!\Gamma(a)\Gamma(b)}(\frac{1}{(a+k)^{J+1}}+\frac{(-1)^m}{(b+k)^{J+1}})
\end{equation*}
If we consider some special cases for instance $[\frac{1}{F(z)}]_{0}$ and $[\frac{1}{F(z)}]_{1}$, then\\
\begin{equation*}
  \Gamma(a)\Gamma(b)=\sum_{k=0}^{\infty}\frac{(-1)^{k}\Gamma(a+b+k)}{k!}(\frac{1}{a+k}+\frac{1}{b+k})
\end{equation*}
and\\
\begin{equation*}
  \Gamma(a)\Gamma(b)(\sum_{k=0}^{\infty} \frac{1}{a+k}-\frac{1}{b+k})=\sum_{k=0}^{\infty}\frac{(-1)^{k}\Gamma(a+b+k)}{k!}(\frac{1}{(a+k)^2}-\frac{1}{(b+k)^2})
\end{equation*}
Moreover, if we let $a=\frac{1}{4}, b=\frac{3}{4}$, then it becomes\\
\begin{equation*}
  \Gamma(\frac{1}{4})\Gamma(\frac{3}{4})(\sum_{k=0}^{\infty} \frac{1}{\frac{1}{4}+k}-\frac{1}{\frac{3}{4}+k})=\sum_{k=0}^{\infty}(-1)^{k}(\frac{1}{(\frac{1}{4}+k)^2}-\frac{1}{(\frac{3}{4}+k)^2})
\end{equation*}
or\\
\begin{equation*}
  \sum_{k=0}^{\infty}(-1)^{\frac{k(k+1)}{2}}\frac{1}{(2k+1)^2}=\frac{\sqrt{2}\pi^2}{16}
\end{equation*}
\end{example}

\section{The formula of derivative}
We can take derivative with respect to $a$ in (\ref{co3.1}). As a simple application, following we discuss the case when $m=2$.\\
\begin{theorem}\label{thm7}
\begin{equation*}
\Gamma(a+z)\Gamma(a-z)(\psi(a+z)+\psi(a-z))=2\psi(a)\Gamma(a)^2+S_{1}(a,z)-S_{2}(a,z)
\end{equation*}
where\\
\begin{align*}
   & S_{1}(a,z)=\sum_{k=0}^{\infty}\sum_{n=1}^{\infty}\frac{4(-1)^k \psi(2a+k)\Gamma(2a+k)}{k!(a+k)^{2n+1}}z^{2n} \\
   & S_{2}(a,z)=\sum_{k=0}^{\infty}\sum_{n=1}^{\infty}\frac{2(2n+1)(-1)^k \Gamma(2a+k)}{k!(a+k)^{2n+2}}z^{2n}
\end{align*}
\end{theorem}
\begin{proof}
The proof is easy, recall (\ref{ex1.2})\\
\begin{equation*}
  \Gamma(a+z)\Gamma(a-z)=\Gamma(a)^2+\sum_{k=0}^{\infty}\frac{2(-1)^{k+1}\Gamma(2a+k)}{(a+k)k!}\frac{z^2}{z^2-(a+k)^2}
\end{equation*}
where it converges in $|z|<a$. The right hand side $\frac{z^2}{z^2-(a+k)^2}$ can be represented as $-\sum_{n=1}^{\infty}(\frac{z}{a+k})^{2n}$. It can be reformulated as\\
\begin{equation*}
\Gamma(a+z)\Gamma(a-z)=\Gamma(a)^2+\sum_{k=0}^{\infty}\sum_{n=1}^{\infty}\frac{2(-1)^{k}\Gamma(2a+k)}{k!(a+k)^{2n+1}}z^{2n}
\end{equation*}
Take $\frac{\partial }{\partial a}$ on both sides then obtain what required.\\
\end{proof}

\begin{corollary}
Denote\\
\begin{align*}
      &\zeta_{A}(s)=\sum_{k=1}^{\infty}\frac{(-1)^{k-1}}{k^s}\\
      &\zeta_{A,H}(s)=\sum_{k=1}^{\infty}\frac{(-1)^{k-1}H_{k}}{k^s}
\end{align*}
where $H_{k}$ is the harmonic number $H_{k}=1+\frac{1}{2}+...+\frac{1}{k}$. Then the following relation is valid for $n\in\mathbb{Z}^{+}$.\\
\begin{equation*}
  \zeta_{A,H}(2n)=(n-\frac{n}{2^{2n}}-\frac{1}{2^{2n+1}})\zeta(2n+1)-\sum_{j=1}^{n-1}\zeta(2j+1)\zeta_{A}(2n-2j)
\end{equation*}
\end{corollary}
\begin{proof}
Let $a=1$ then according to Theorem \ref{thm7}\\
\begin{equation*}
\frac{\pi z}{\sin (\pi z)}(\psi(1+z)+\psi(1-z))=-2\gamma+S_{1}(1,z)-S_{2}(1,z)
\end{equation*}
\begin{align*}
  S_{1}(1,z) & =\sum_{k=0}^{\infty}\sum_{n=1}^{\infty}\frac{4(-1)^k \psi(2+k)}{(1+k)^{2n}}z^{2n}\\
   & =-\gamma\sum_{k=0}^{\infty}\sum_{n=1}^{\infty}\frac{4(-1)^{k}}{(1+k)^{2n}}z^{2n}+\sum_{k=0}^{\infty}\sum_{n=1}^{\infty}\frac{4(-1)^k H_{1+k}}{(1+k)^{2n}}z^{2n}
\end{align*}
or rewritten as\\
\begin{equation*}
[S_{1}(1,z)]_{2n}=-4\gamma\zeta_{A}(2n)+4\zeta_{A,H}(2n), (n\geq 1)
\end{equation*}
On the other hand\\
\begin{align*}
  S_{2}(a,z)&=\sum_{k=0}^{\infty}\sum_{n=1}^{\infty}\frac{2(2n+1)(-1)^k \Gamma(2a+k)}{k!(a+k)^{2n+2}}z^{2n}\\
  &=\sum_{n=1}^{\infty}\sum_{k=0}^{\infty}\frac{2(2n+1)(-1)^k}{(k+1)^{2n+1}}z^{2n}
\end{align*}
or rewritten as\\
\begin{equation*}
[S_{2}(1,z)]_{2n}=2(2n+1)\zeta_{A}(2n+1), (n\geq 1)
\end{equation*}
Therefore\\
\begin{equation*}
[S_{1}(1,z)-S_{2}(1,z)]_{2n}=-4\gamma\zeta_{A}(2n)+4\zeta_{A,H}(2n)-2(2n+1)\zeta_{A}(2n+1), (n\geq 1)
\end{equation*}
If we denote $[\frac{\pi z}{\sin (\pi z)}]_{n}$ by $\tau_{n}$, and note that $(\psi(1+z)+\psi(1-z))=-2(\gamma+\sum_{n=1}^{\infty}\zeta(2n+1)z^{2n})$ then\\
\begin{equation*}
[\frac{\pi z}{\sin (\pi z)}(\psi(1+z)+\psi(1-z))]_{2n}=-2(\gamma\tau_{2n}+\zeta(2n+1)+\sum_{j=1}^{n-1}\zeta(2j+1)\tau_{2n-2j})
\end{equation*}
Therefore\\
\begin{equation*}
2\gamma\zeta_{A}(2n)-2\zeta_{A,H}(2n)+(2n+1)\zeta_{A}(2n+1)=\gamma\tau_{2n}+\zeta(2n+1)+\sum_{j=1}^{n-1}\zeta(2j+1)\tau_{2n-2j}
\end{equation*}
One can show that $\tau_{n}=2\zeta_{A}(2n)$ and $\zeta_{A}(s)=(1-\frac{1}{2^{s-1}})\zeta(s)$. It follows that\\
\begin{equation*}
  \zeta_{A,H}(2n)=(n-\frac{n}{2^{2n}}-\frac{1}{2^{2n+1}})\zeta(2n+1)-\sum_{j=1}^{n-1}\zeta(2j+1)\zeta_{A}(2n-2j)
\end{equation*}
\end{proof}

\begin{corollary}
Denote\\
\begin{align*}
      &\beta(s)=\sum_{k=0}^{\infty}\frac{(-1)^{k}}{(2k-1)^s}\\
      &\beta_{H}(s)=\sum_{k=1}^{\infty}\frac{(-1)^{k-1}H_{k-1}}{(2k-1)^s}
\end{align*}
Then for $n\in\mathbb{Z}^{+}$ the following relation is valid.\\
\begin{equation*}
  \beta_{H}(2n+1)=(2n+1)\beta(2n+2)-\beta(2n+1)\log 4-2\sum_{j=0}^{n-1}(1-\frac{1}{2^{2n-2j+1}})\beta(2j+1)\zeta(2n-2j+1)
\end{equation*}
\end{corollary}
\begin{proof}
Let $u_{2n}=[\psi(\frac{1}{2}+z)+\psi(\frac{1}{2}-z)]_{2n}$, one can show that\\
\begin{equation*}
  u_{2n}=\begin{cases}
  2(1-2^{2n+1})\zeta(2n+1),\text{if } n\geq 1\\
  2(-\gamma-\log 4),\text{if } n=0
  \end{cases}
\end{equation*}
On the other hand,\\
\begin{equation*}
\frac{\pi}{\cos(\pi z)}=\sum_{k=0}^{\infty}\frac{(-1)^{k}E_{2k}\pi^{2k+1}}{(2k)!}z^{2k}
\end{equation*}
where $E_{2k}$ is Euler number $E_{0}=1,E_{2}=-1,E_{4}=5,E_{6}=-61,...$. Recall the formula about the Dirichlet beta function $\beta(2n+1)=\frac{(-1)^{n}E_{2n}\pi^{2n+1}}{2^{2n+2}(2n)!}$ we have\\
\begin{equation*}
  [\frac{\pi}{\cos(\pi z)}]_{2n}=2^{2n+2}\beta(2n+1)
\end{equation*}
One can show that
\begin{align*}
   & S_{1}(\frac{1}{2},z)=\sum_{k=0}^{\infty}\sum_{n=1}^{\infty}\frac{4(-1)^k \psi(1+k)}{(\frac{1}{2}+k)^{2n+1}}z^{2n} \\
   & S_{2}(\frac{1}{2},z)=\sum_{k=0}^{\infty}\sum_{n=1}^{\infty}\frac{2(2n+1)(-1)^k }{(\frac{1}{2}+k)^{2n+2}}z^{2n}
\end{align*}
or rewritten as\\
\begin{align*}
   & [S_{1}(\frac{1}{2},z)]_{2n}=2^{2n+3}(-\gamma \beta(2n+1)+\beta_{H}(2n+1)) \\
   & [S_{2}(\frac{1}{2},z)]_{2n}=2^{2n+3}(2n+1)\beta(2n+2)
\end{align*}
According to Theorem \ref{thm7}, by some computation one can get\\
\begin{equation*}
  \beta_{H}(2n+1)=(2n+1)\beta(2n+2)-\beta(2n+1)\log 4-2\sum_{j=0}^{n-1}(1-\frac{1}{2^{2n-2j+1}})\beta(2j+1)\zeta(2n-2j+1)
\end{equation*}
\end{proof}

\section{Differential Relation}
\begin{theorem}\label{thm8}(\textbf{Differential Relation})\\
Suppose that the sequence $(a_{k})$ satisfies the condictions in Theorem \ref{thm2}. Let $F(z)=z\prod_{k=0}^{\infty}(1-(\frac{z}{a_{k}})^2)$, $H(z)=\frac{F'(z)}{F(z)}$, then\\
\begin{equation*}
  \frac{F''(z)}{F(z)}=H'(z)+H^2(z)=\frac{F''(\infty)}{F(\infty)}+\sum_{k=0}^{\infty}\frac{F''(a_{k})}{F'(a_{k})}\frac{2a_{k}}{z^2-a_{k}^{2}}
\end{equation*}
where
\begin{equation*}
\frac{F''(\infty)}{F(\infty)}=\sum_{k=0}^{\infty}\frac{2F''(a_{k})}{a_{k}F'(a_{k})}-6\sum_{k=0}^{\infty}\frac{1}{a_{k}^2}
\end{equation*}
Equivalently, the relation can be reformulated as\\
\begin{equation*}
    [\frac{F''(z)}{F(z)}]_{J}=\begin{cases}
    0,\text{ if } J \text{ odd}\\
    -6\sum_{k=0}^{\infty}\frac{1}{a_{k}2}, \text{ if } J=0\\
    \sum_{k=0}^{\infty}\frac{-2F''(a_{k})}{a_{k}^{J+1}F'(a_{k})}, \text{ if } J=2,4,6,...
    \end{cases}
\end{equation*}
\end{theorem}
\begin{proof}
Note that $H=(\log F(z))'$. It turns out that\\
\begin{equation}\label{thm8.1}
  H(z)=\frac{1}{z}+\sum_{k=0}^{\infty}\frac{2z}{z^2-a_{k}^2}
\end{equation}
For convenience define $\widetilde{H}(z):=H(z)-\frac{1}{z}$, by the convergence of $\sum \frac{1}{|a_{k}|^2}$ one can see that $\widetilde{H}(z)$ converges normally in $|z|<M$ By some easy computation one has\\
\begin{equation*}
\widetilde{H}'(z)=\frac{1}{z}\widetilde{H}(z)-\sum_{k=0}^{\infty}(\frac{2z}{z^2-a_{k}^2})^2
\end{equation*}
On the other hand, by (\ref{thm8.1})\\
\begin{equation*}
\widetilde{H}(z)^2=\sum_{k=0}^{\infty}(\frac{2z}{z^2-a_{k}^2})^2+\sum_{i\neq j,i,j\geq 0}\frac{2z}{z^2-a_{i}^2}\frac{2z}{z^2-a_{j}^2}
\end{equation*}
Therefore the following relation is valid\\
\begin{equation*}
\widetilde{H}'(z)+\widetilde{H}(z)^2-\frac{1}{z}\widetilde{H}(z)=\Delta
\end{equation*}
where\\
\begin{equation*}
\Delta=\sum_{i\neq j,i,j\geq 0}\frac{2z}{z^2-a_{i}^2}\frac{2z}{z^2-a_{j}^2}
\end{equation*}
Taking the partial fraction decomposition, then\\
\begin{equation*}
\frac{\Delta}{4z^2}=\sum_{k=0}^{\infty}\widetilde{\delta}_{k}\frac{1}{z^2-a_{k}^2}
\end{equation*}
where\\
\begin{equation*}
  a_{k}\widetilde{\delta}_{k}=\sum_{s\neq k,s=0}^{\infty}\frac{2a_{k}}{a_{k}^2-a_{s}^2}=\lim_{z\rightarrow a_{k}}(\widetilde{H}(z)-\frac{2z}{z^2-a_{k}^2})=\frac{F''(a_{k})}{2F'(a_{k})}-\frac{3}{2a_{k}}
\end{equation*}
Hence\\
\begin{equation*}
  \widetilde{H}'(z)+\widetilde{H}(z)^2-\frac{1}{z}\widetilde{H}(z)=\sum_{k=0}^{\infty}(\frac{F''(a_{k})}{2a_{k}F'(a_{k})}-\frac{3}{2a_{k}^2})\frac{4z^2}{z^2-a_{k}^2}
\end{equation*}
or reformulated as\\
\begin{equation*}
  \frac{F''(z)}{F(z)}-\frac{3F'(z)}{zF(z)}+\frac{3}{z^2}=\sum_{k=0}^{\infty}(\frac{F''(a_{k})}{2a_{k}F'(a_{k})}-\frac{3}{2a_{k}^2})\frac{4z^2}{z^2-a_{k}^2}
\end{equation*}
Moreover, after some simplification, one has a more beautiful relation
\begin{equation}\label{thm8.2}
  \frac{F''(z)}{F(z)}=\frac{F''(\infty)}{F(\infty)}+\sum_{k=0}^{\infty}\frac{F''(a_{k})}{F'(a_{k})}\frac{2a_{k}}{z^2-a_{k}^{2}}
\end{equation}
where
\begin{equation*}
\frac{F''(\infty)}{F(\infty)}=\sum_{k=0}^{\infty}\frac{2F''(a_{k})}{a_{k}F'(a_{k})}-6\sum_{k=0}^{\infty}\frac{1}{a_{k}^2}
\end{equation*}
If we expand (\ref{thm8.2}) as power series around $0$, we can find that\\
\begin{equation*}
    [\frac{F''(z)}{F(z)}]_{J}=\begin{cases}
    0,\text{ if } J \text{ odd}\\
    -6\sum_{k=0}^{\infty}\frac{1}{a_{k}2}, \text{ if } J=0\\
    \sum_{k=0}^{\infty}\frac{-2F''(a_{k})}{a_{k}^{J+1}F'(a_{k})}, \text{ if } J=2,4,6,...
    \end{cases}
\end{equation*}
\end{proof}

\begin{corollary}\label{cor7}
If we define $c_{m}=\sum_{k=0}^{\infty}\frac{1}{a_{k}^{2m}}$ and $\delta_{k}=\frac{F''(a_{k})}{F'(a_{k})}$, then the following relation holds for $m=1,2,3,...$\\
\begin{equation*}
(2m+3)c_{m+1}-2\sum_{i=1}^{m}c_{i}c_{m+1-i}=\sum_{k=0}^{\infty}\frac{\delta_{k}}{a_{k}^{2m+1}}
\end{equation*}
\end{corollary}
\begin{proof}
In Theorem \ref{thm8}\\
\begin{equation*}
H'(z)+H^2(z)=\frac{F''(z)}{F(z)}
\end{equation*}
For $H(z)=\frac{1}{z}+\sum_{k=1}^{\infty}\frac{2z}{z^2-a_{k}^2}$, there is an expansion\\
\begin{equation*}
  H(z)=\frac{1}{z}-2\sum_{n=1}^{\infty}c_{n}z^{2n-1}
\end{equation*}
By the rearrangement we have\\
\begin{equation*}
  H'(z)+H^2(z)=-6c_{1}+\sum_{m=1}^{\infty}\widetilde{C}_{m}z^{2m}
\end{equation*}
where\\
\begin{equation*}
  \widetilde{C}_{m}=-2(2m+3)c_{m+1}+4\sum_{i=1}^{m}c_{i}c_{m+1-i}
\end{equation*}
That is\\
\begin{equation*}
    [H'(z)+H^2(z)]_{m}=\begin{cases}
0,\text{ if } m \text{ odd}\\
    -6c_{1}, \text{ if } m=0\\
    -2(2m+3)c_{m+1}+4\sum_{i=1}^{m}c_{i}c_{m+1-i}, \text{ if } m=2,4,6,...
    \end{cases}
\end{equation*}
According to Theorem \ref{thm8}\\
\begin{equation*}
  [\frac{F''(z)}{F(z)}]_{J}=\begin{cases}
    0,\text{ if } J \text{ odd}\\
    -6\sum_{k=0}^{\infty}\frac{1}{a_{k}2}, \text{ if } J=0\\
    \sum_{k=0}^{\infty}\frac{-2\delta_{k}}{a_{k}^{J+1}}, \text{ if } J=2,4,6,...
    \end{cases}
\end{equation*}
By comparing the coefficients one obtain the conclusion.
\end{proof}

\begin{example}
A special case is when $F(z)=\pi \sin(\pi z)$. It follows from $F(z)=z\prod_{k=0}^{\infty}(1-(\frac{z}{1+k})^2)=\pi\sin(\pi z)$ one has $\frac{F''(z)}{F(z)}=-\pi^2$ where $a_{k}=1+k$. Therefore $\delta_{k}=\frac{F''(a_{k})}{F'(a_{k})}=0$. If we denote $\zeta(2n)$ by $c_{n}$, then according to Corollary \ref{cor7}\\
\begin{equation*}
(2m+3)c_{m+1}-2\sum_{i=1}^{m}c_{i}c_{m+1-i}=0
\end{equation*}
or rewritten as
\begin{equation*}
(2m+3)\zeta(2m+2)=2\sum_{i=1}^{m}\zeta(2i)\zeta(2m+2-2i)
\end{equation*}
This is a well-known recursion relation for $\zeta(2n)$.
\end{example}

Following example provides a similar relation for $\zeta(2n,a)$\\
\begin{example}
Let $F(z)=z\prod_{k=0}^{\infty}(1-(\frac{x}{a+k})^2)=\frac{z\Gamma(a)^2}{\Gamma(a+z)\Gamma(a-z)}$. It's routine to to find $\frac{F''(z)}{F'(z)}$. Here  $a_{k}=a+k$.\\
\begin{equation*}
\frac{F''(z)}{F'(z)}=\frac{z(\psi(a-z)-\psi(a+z))^2+2(\psi(a-z)-\psi(a+z))-z(\psi'(a-z)+\psi'(a+z))}{1+z(\psi(a-z)-\psi(a+z))}
\end{equation*}
therefore\\
\begin{align*}
\frac{F''(a_{k})}{F'(a_{k})}&=\frac{(a+k)(\psi(-k)-\psi(2a+k))^2+2(\psi(-k)-\psi(2a+k))-(a+k)(\psi'(-k)+\psi'(2a+k))}{1+(a+k)(\psi(-k)-\psi(2a+k))}\\
&=\frac{(a+k)(\psi(-k)^2-\psi'(-k))+(2-2(a+k)\psi(2a+k))\psi(-k)+M_{1}}{(a+k)\psi(-k)+M_{2}}\\
&=\lim_{z\rightarrow -k}\frac{\psi(z)^2-\psi'(z)}{\psi(z)}+\frac{2}{a+k}-2\psi(2a+k)
\end{align*}
Note that
\begin{equation*}
\lim_{z\rightarrow 0}\frac{\psi(z)^2-\psi'(z)}{\psi(z)}=-2\gamma
\end{equation*}
and recall$\psi(z+k)=\psi(z)+\frac{1}{z}+\frac{1}{z+1}+...+\frac{1}{z+k-1}$. Hence\\
\begin{equation*}
\lim_{z\rightarrow -k}\frac{\psi(z)^2-\psi'(z)}{\psi(z)}=2(H_{k}-\gamma)
\end{equation*}
Therefore\\
\begin{align*}
\delta_{k}=\frac{F''(a_{k})}{F'(a_{k})}&=2(H_{k}+\frac{1}{a+k}-\gamma-\psi(2a+k))\\
&=2(\psi(1+k)+\frac{1}{a+k}-\psi(2a+k))
\end{align*}
This makes
\begin{equation*}
\delta_{k}=\frac{F''(a_{k})}{F'(a_{k})}=2(H_{k}+\frac{1}{a+k}-\gamma-\psi(2a+k))=2(\psi(1+k)+\frac{1}{a+k}-\psi(2a+k))
\end{equation*}
\end{example}
Therefore we have the conclusion according to Corollary \ref{cor7}:\\
\begin{theorem}(\textbf{recursion formula for $\zeta(2m,a)$})
For $m=1,2,3,...$ we have\\
\begin{equation*}
    (m+\frac{1}{2})\zeta(2m+2,a)-\sum_{i=1}^{m}\zeta(2i,a)\zeta(2m-2i+2,a)=\sum_{k=0}^{\infty}\frac{\psi(1+k)-\psi(2a+k)}{(a+k)^{2m+1}}
\end{equation*}
\end{theorem}

Observe the formula (\ref{thm8.2}), we conjecture that there is a more general relation\\
\begin{conjecture}
\begin{equation*}
  \frac{F^{(L)}(z)}{F(z)}=\frac{F^{(L)}(\infty)}{F(\infty)}+\sum_{k=0}^{\infty}\sum_{r=0}^{m-1}\frac{F^{(L)}(a_{k})}{F'(a_{k})}(\frac{\omega_{m}^{r}}{z-a_{k}\omega_{m}^{r}})  
\end{equation*}
\end{conjecture}


\bibliographystyle{unsrt}  


Xiaowei Wang(\begin{CJK}{UTF8}{gbsn}王骁威\end{CJK})\\
Institut f\"{u}r Mathematik, Universit\"at  Potsdam, Potsdam OT Golm, Germany\\
 Email: \texttt{xiawang@gmx.de}

\end{document}